\documentclass{amsart}

\usepackage{amssymb,amsmath,amscd,amsthm,enumerate}
\usepackage{color}

\date{\today}

\newcommand{\A}{{\mathcal A}}

\newcommand{\Z}{{\mathbb Z}}
\newcommand{\R}{{\mathbb R}}
\newcommand{\C}{{\mathbb C}}

\newcommand{\T}{{\mathbb T}}
\newcommand{\Q}{{\mathbb Q}}
\newcommand{\D}{{\mathbb D}}

\newcommand{\E}{{\mathcal E}}

\newcommand{\ZL}{{\mathcal Z}}

\newtheorem{theorem}{Theorem}[section]
\newtheorem{lemma}[theorem]{Lemma}

\theoremstyle{definition}
\newtheorem{remark}[theorem]{Remark}
\theoremstyle{definition}

\theoremstyle{definition}

\allowdisplaybreaks

\numberwithin{equation}{section}

\newcommand{\tr}{\mathrm{tr} }

\newcommand{\set}[1]{\left\{#1\right\}}
\newcommand{\eqdef}{\overset{\mathrm{def}}=}

\begin{document}

\title[Gordon Lemmas for CMV Matrices]{Purely Singular Continuous Spectrum for Sturmian CMV Matrices via Strengthened Gordon Lemmas}

\author[J.\ Fillman]{Jake Fillman}

\address{Department of Mathematics, Virginia Tech, 225 Stanger Street, Blacksburg, VA~24061, USA}

\email{fillman@vt.edu}

\begin{abstract}

The Gordon Lemma refers to a class of results in spectral theory which prove that strong local repetitions in the structure of an operator preclude the existence of eigenvalues for said operator. We expand on recent work of Ong and prove versions of the Gordon Lemma that are valid for CMV matrices and which do not restrict the parity of scales upon which repetitions occur. The key ingredient in our approach is a formula of Damanik-Fillman-Lukic-Yessen which relates two classes of transfer matrices for a given CMV operator. There are  many examples to which our result can be applied. We apply our theorem to complete the classification of the spectral type of CMV matrices with Sturmian Verblunsky coefficients; we prove that such CMV matrices have purely singular continuous spectrum supported on a Cantor set of zero Lebesgue measure for all (irrational) frequencies and all phases. We also discuss applications to CMV matrices with Verblunsky coefficients generated by general codings of rotations.

\end{abstract}

\maketitle




\section{Introduction}

In this paper, we will discuss the spectral type of aperiodic extended CMV matrices. Extended CMV matrices arise as a natural playground for spectral theoretic techniques, as they furnish canonical unitary analogs of Schr\"odinger operators and Jacobi matrices. Moreover, they are interesting in their own right, since they naturally arise in connection with quantum walks in one dimension \cite{CGMV,DEFHV,DFO, DFV}, the classical Ising model \cite{DFLY1, DMY2}, and orthogonal polynomials on the unit circle \cite{S1,S2}. More specifically, given a sequence $\alpha = (\alpha_n)_{n \in \Z}$ with $\alpha_n \in \D = \set{z \in \C : |z| < 1}$, the corresponding extended CMV matrix is given by $\E = \mathcal L \mathcal M$, where
$$
\mathcal L
=
\bigoplus_{j\in\Z} \Theta(\alpha_{2j}),
\quad
\mathcal M
=
\bigoplus_{j \in \Z} \Theta(\alpha_{2j+1}),
\quad
\Theta(\alpha)
=
\begin{pmatrix}
\overline\alpha & \rho \\
\rho & - \alpha
\end{pmatrix},
\quad
\rho
=
\sqrt{1-|\alpha|^2},
$$
and $\Theta(\alpha_n)$ acts on $\ell^2\big(\{n, n+1\} \big)$. A straightforward calculation reveals that $\E$ enjoys the matrix representation
\begin{equation} \label{def:extcmv}
\E
=
\begin{pmatrix}
\ddots & \ddots & \ddots & \ddots &&&&&  \\
  & \overline{\alpha_2}\rho_1 & -\overline{\alpha_2}\alpha_1 & \overline{\alpha_3} \rho_2 & \rho_3\rho_2 &&& \\
& \rho_2\rho_1 & -\rho_2\alpha_1 & -\overline{\alpha_3}\alpha_2 & -\rho_3\alpha_2 &&&  \\
&&& \overline{\alpha_4} \rho_3 & -\overline{\alpha_4}\alpha_3 & \overline{\alpha_5}\rho_4 & \rho_5\rho_4 & \\
&&& \rho_4\rho_3 & -\rho_4\alpha_3 & -\overline{\alpha_5}\alpha_4 & -\rho_5 \alpha_4 &  \\
&&&& \ddots & \ddots &  \ddots & \ddots
\end{pmatrix},
\end{equation}
where $\rho_n = \left( 1 - |\alpha_n|^2 \right)^{1/2}$ for every $n \in \Z$. We refer to $(\alpha_n)_{n \in \Z}$ as the sequence of \emph{Verblunsky coefficients} of $\E$. Note that all unspecified matrix entries are zero.

 In general, one hopes to study the direct spectral problem for such matrices and prove theorems that establish the spectral type for a class of sequences of Verblunsky coefficients -- for example, periodic coefficients always have purely absolutely continuous spectral type, and random coefficients (almost surely) have purely pure point spectral type.\footnote{Even in the random case, the spectral picture here is still somewhat unresolved. To the best of the author's knowledge, it is still not known that a CMV matrix whose Verblunsky coefficients are iid random variables with finitely supported single-site distributions (almost surely) exhibits pure point spectrum.} Recently, aperiodic dynamically defined models for which $\alpha_n$ assumes only finitely many values have been studied \cite{ADZ,DFO, DFV, damaniklenz:szegob, DMY1, DMY2, O14}. These models are popular as they furnish mathematical models of one-dimensional quasicrystals. In general, such models have a strong tendency to exhibit purely singular continuous spectral type. Since operators of the type just described always have empty absolutely continuous spectrum by Kotani theory (see \cite[Section~10.11]{S2}), the relevant challenge is to preclude the existence of eigenvalues. This is typically done with ``Gordon-type'' arguments which exclude eigenvalues by using resonances. A bit more specifically, one uses repetitions within the structure of the operator to build up resonances in the associated generalized eigenfunctions and thus preclude $\ell^2$ decay of the same. This idea was introduced by Gordon in \cite{Gordon76} and has subsequently proved quite fruitful in the analysis of Schr\"odinger, Sturm-Liouville, and Jacobi operators; see \cite{DamGordon, Damanik2004, DamanikStolz2000, Gordon86, Kruger, Seif2011, Seif2012, Seif2014, SeifVogt2014} for a few examples.

This philosophy was partially extended to CMV matrices by Ong in \cite{O12, O13, O14}, with the caveat that one could only use scales of repetition of even length. At the time, this appeared to be an intrinsic drawback to CMV matrices, arising from the readily apparent fact that even and odd Verblunsky coefficients are not on precisely equal footing vis-\`a-vis the generalized eigenvalue equation $\E \psi = z\psi$, an asymmetry that is manifested in the alternating structure of the Gesztesy--Zinchenko transfer matrices (which propagate solution data for the eigenvalue equation and will be defined presently). However, in a recent paper of Damanik-Fillman-Lukic-Yessen, it has been observed that a two-step GZ transfer matrix can be related to a two-step Szeg\H{o} transfer matrix in a simple fashion \cite{DFLY2}, which enables us to remove the parity restrictions in the CMV Gordon Lemmas by passing to Szeg\H{o} transfer matrices, which do treat even and odd Verblunsky coefficients on equal footing.

\subsection{Sturmian CMV Matrices}

As a prime example of the utility of the new Gordon Lemmas, we discuss extended CMV matrices whose Verblunsky coefficients are \emph{Sturmian}. We begin by very briefly recalling some of the relevant background; \cite{loth1, loth2, queff} are inspired references for the interested reader. Given an alphabet with two letters, say $\A = \set{0,1}$, consider a sequence $\omega \in \A^{\Z}$. The \emph{factor complexity function} of $\omega$ is defined by
$$
p_\omega(n)
=
\# F_\omega(n),
\quad n \geq 0,
$$
where $F_\omega(n)$ denotes the set of all distinct subwords of $\omega$ of length $n$. The celebrated theorem of Hedlund and Morse states that a recurrent sequence $\omega \in \A^{\Z}$ is periodic if and only if $p_\omega(n) \leq n$ for some $n \in \Z_+$ \cite{MH38}. We say that $\omega$ is \emph{Sturmian} if it is recurrent and satisfies $p_\omega(n) = n+1$ for every $n \geq 0$. Note that considering $n = 1$ already forces the underlying alphabet to have precisely two symbols. In light of the aforementioned Hedlund-Morse Theorem, Sturmian sequences are precisely those aperiodic sequences that are ``closest to periodic'' in the sense of combinatorial complexity.  Equivalently, it turns out that $\omega$ is Sturmian if and only if it is a so-called mechanical sequence of irrational slope. More precisely, $\omega$ is Sturmian if and only if there exist $\theta \in (0,1) \setminus \Q$ and $\varphi \in [0,1)$ such that $\omega = s_{\theta,\varphi}$ or $\omega = s_{\theta,\varphi}'$, where
\begin{align*}
s_{\theta,\varphi}(n)
& =
\lfloor (n+1)\theta + \varphi \rfloor
-
\lfloor n\theta + \varphi \rfloor,\\
s_{\theta,\varphi}'(n)
& =
\lceil (n+1)\theta + \varphi \rceil
-
\lceil n\theta + \varphi\rceil
\end{align*}
for $n \in \Z$. We call $\theta$ the \emph{frequency} and $\varphi$ the \emph{phase} of $\omega$. The set of all Sturmian words with a given frequency $\theta$ is called the associated \emph{subshift}:
\[
\Omega
=
\Omega_\theta
=
\{ s_{\theta,\varphi} : \varphi \in [0,1) \}
\cup
\{ s_{\theta,\varphi}' : \varphi \in [0,1) \}
\subseteq
\A^{\Z}.
\]
It is straightforward to check that $\Omega_\theta$ is a compact subset of $\A^{\Z}$ equipped with the product topology and that $\Omega$ is invariant under the action of the left shift $S:\A^{\Z} \to \A^{\Z}$ defined by $(S\omega)_n = \omega_{n+1}$. Additionally, $(\Omega_\theta,S)$ is a strictly ergodic dynamical system; that is, it is both minimal and uniquely ergodic. By a Sturmian CMV matrix, we shall then mean one whose Verblunsky coefficients assume two distinct values, modulated by a specific Sturmian sequence. More precisely, we choose $\beta \neq \gamma$  in $\D$, an irrational frequency $\theta \in (0,1) \setminus \Q$, and a Sturmian sequence $\omega \in \Omega_\theta$, and then define an extended CMV matrix $\E = \E_{\omega}$ by
\begin{equation} \label{eq:sturmalphas}
\alpha_\omega(n)
=
\beta + \omega_n(\gamma-\beta)
=
\begin{cases}
\beta & \omega_n = 0 \\
\gamma  & \omega_n = 1
\end{cases}
\end{equation}

\begin{theorem} \label{t:sturm:scspec}
For all $\beta \neq \gamma$ in $\D$ and all Sturmian sequences $\omega = s_{\theta,\varphi}$ or $\omega = s_{\theta,\varphi}' \in \A^{\Z}$, $\E_{\omega}$ has purely singular continuous spectrum supported on a Cantor set of zero Lebesgue measure.
\end{theorem}

It is well-known that the spectrum $\Sigma = \sigma(\E_{\omega})$ is independent of $\omega \in \Omega_\theta$, i.e., $\Sigma$ depends only on $\beta$, $\gamma$, and $\theta$. Moreover, $\Sigma$ is a Cantor set of zero Lebesgue measure for all $\beta \neq \gamma$, and Sturmian $\omega$, and hence, it cannot support absolutely continuous measures \cite[Theorem~1.1]{damaniklenz:szegob}; see also \cite{damaniklenz:b}. Consequently, the new content of Theorem~\ref{t:sturm:scspec} is supplied by proving that $\E_{\omega}$ does not have eigenvalues. Let us note that this extends the results of \cite{O14} in two directions. First, we do not need to restrict the continued fraction coefficients of the frequency $\theta$. Secondly, \cite{O14} makes a statement for Lebesgue-almost every phase $\varphi \in [0,1)$, while this result describes the spectral type for arbitrary phase.

\subsection{Codings of Rotations} We may also apply our new Gordon Lemmas to CMV matrices whose Verblunsky coefficients are generated by more general codings of rotations. Given $\theta \in (0,1) \setminus \Q$, $\varphi \in [0,1)$, and a non-degenerate interval $I \subseteq \T \eqdef \R / \Z$, the associated \emph{rotation coding sequence} $u \in \{0,1\}^\Z$ is given by
\begin{equation} \label{eq:coddef}
u(n)
=
\chi_I(n\theta + \varphi).
\end{equation}
Notice that Sturmian sequences are a special case of rotation codings, obtained by setting $I = [1-\theta,1)$ or $I = (1-\theta,1]$. As with Sturmian sequences, one can use codings of rotations to generate CMV matrices; more specifically, given $\beta \neq \gamma$ in $\D$, and $\theta, \varphi$, and $I$ as above, we consider $\E = \E_{\theta,\varphi}^I$ given by
\[
\alpha_{\theta,\varphi}^I(n)
=
\begin{cases} \beta & u(n) = 0\\ \gamma & u(n) = 1 
\end{cases}
\] 
where $u_n$ is defined by \eqref{eq:coddef}. Our theorem for these is precisely \cite[Theorem~4]{O14} with the hypothesis that $\theta$ have infintely many even continued fraction denominators removed.

\begin{theorem} \label{t:rotcode}
Let $\theta \in (0,1) \setminus \Q$ have continued fraction expansion
\[
\theta
=
\cfrac{1}{a_1 + \cfrac{1}{a_2 + \cfrac{1}{\ddots}}},
\]
with $a_j \in \Z_+$ for all $j$, and fix a nondegenerate interval $I \subseteq \T$. If $\limsup_{j \to \infty} a_j \geq 4$, then $\E_{\theta,\varphi}^I$ has purely continuous spectrum for Lebesgue almost-every $\varphi \in [0,1)$.
\end{theorem}

\subsection{Gordon Lemmas for CMV Matrices}

Let us now precisely describe the new CMV Gordon Lemmas. Given $\alpha \in \D$ and $z \in \partial \D$, the corresponding \emph{Szeg\H{o} matrix} is defined by
$$
S(\alpha,z)
=
\frac{1}{\rho} 
\begin{pmatrix}
z & - \overline \alpha \\
-\alpha z & 1
\end{pmatrix},
\quad
\rho = \rho_\alpha = \sqrt{1 - |\alpha|^2}.
$$
The Szeg\H{o} transfer matrices associated to $\E_\alpha$ are then defined by
$$
T(n,m;z)
=
\begin{cases}
S(\alpha_{n-1},z) \cdots S(\alpha_m,z) & n > m \\
I & n = m \\
T(m,n;z)^{-1} & n < m
\end{cases}
$$
for $n, m \in \Z$ and $z \in \partial \D$. We will also consider the Gesztesy--Zinchenko transfer matrices \cite{GZ06}. Define
\begin{equation} \label{gz:onestepmats:def}
P(\alpha,z)
=
\frac{1}{\rho}
\begin{pmatrix}
-\alpha & z^{-1} \\
z & - \overline \alpha
\end{pmatrix}
,
\quad
Q(\alpha,z)
=
\frac{1}{\rho}
\begin{pmatrix}
-\overline\alpha & 1 \\
1 & - \alpha
\end{pmatrix},
\, \alpha \in \D,\,  z \in \C \setminus \{0\},
\end{equation}
where $\rho = \left( 1 - |\alpha|^2\right)^{1/2}$ as before. One may use $P$ and $Q$ to capture the recursion described by the difference equation $\E u = zu$ in the following sense. If $u \in \C^{\Z}$ is such that $\E u = zu$, then we define $v = \mathcal M u$, and since $\Theta(\alpha)$ is real-symmetric for all $\alpha \in \D$, we have
\[
\E^\top \! v 
=
\mathcal M \mathcal L \mathcal M u
=
z v.
\] 
Now, with $\Phi(n) = (u_n, v_n)^\top$, Gesztesy and Zinchenko's arguments show that
\begin{equation} \label{eq:gz:stepbystep}
\Phi(n+1)
=
\begin{cases}
P(\alpha_n,z)
\Phi(n)
& n \text{ is even} \\
Q(\alpha_n,z)
\Phi(n)
& n \text{ is odd}
\end{cases}
\end{equation}
for all $n \in \Z$; see also \cite[Proposition~3]{MSB} for a  version of this formalism for more general finite-width unitary operators. It is relatively easy to see how these matrices arise from the definitions of $u$ and $v$. In particular, \eqref{eq:gz:stepbystep} is exactly what one observes if one solves
$$
\Theta(\alpha_{2j-1})
\begin{pmatrix} u_{2j-1} \\ u_{2j} \end{pmatrix}
=
\begin{pmatrix} v_{2j-1} \\ v_{2j} \end{pmatrix},
\quad
\Theta(\alpha_{2j})
\begin{pmatrix} v_{2j} \\ v_{2j+1} \end{pmatrix}
=
z \begin{pmatrix} u_{2j} \\ u_{2j+1} \end{pmatrix}
$$
for $u_{n+1}$ and $v_{n+1}$ in terms of $v_n$ and $u_n$. Notice that we have shifted Gesztesy--Zinchenko's defintions slightly so that our conventions for the definition of $\E$ line up with those of \cite{S1,S2}. In light of these observations, we denote $Y(n,z) = P(\alpha_n,z)$ when $n$ is even and $Y(n,z) = Q(\alpha_n,z)$ when $n$ is odd, and we define the \emph{Gesztesy-Zinchenko cocycle} by
\begin{equation} \label{gz:def}
Z(n,m;z)
=
\begin{cases}
Y(n - 1,z) \cdots Y(m, z) & n > m \\
I & n = m \\
Z(m,n;z)^{-1} & n < m
\end{cases}
\end{equation}
If $u$, $v$, and $\Phi$ are as above, we have
\begin{equation} \label{eq:gzsoltransfer}
\Phi(n)
=
Z(n, m; z) \Phi(m)
\text{ for all } n,m \in \Z
\end{equation}
by \eqref{eq:gz:stepbystep}; see also \cite[Lemma~2.2]{GZ06}. We will typically abbreviate and refer to $Z(n,m;z)$ as a GZ matrix.

\bigskip

Each family of cocycles has distinct applications, advantages, and drawbacks. The Szeg\H{o}-type matrices are easier to work with in general, since every one-step Szeg\H{o} matrix has the same structure, whereas the one-step GZ matrices alternate. However, the GZ matrices propagate solution data for the eigenvalue equation, which is more useful from the point of view of spectral theory, since one can translate information about solutions into information about resolvents rather directly. Delightfully, there is a rather straightforward manner in which the two families of matrices are related, which was observed in \cite{DFLY2}:
\[
Q(\alpha,z) P(\beta,z)
=
z^{-1} S(\alpha,z) S(\beta,z)
\]
for all $\alpha,\beta \in \D$ and $z \in \C \setminus \{0\}$. In particular, one has
\begin{equation} \label{eq:multistep:sgz}
z^{-n} T(2n,0;z)
=
Z(2n,0;z)
\end{equation}
for all $z \in \C \setminus \{0\}$ and all $n \geq0$. This relationship is the key that enables us to prove new versions of the two- and three-block Gordon lemmas for CMV matrices with no restrictions on the parity of the scales of repetitions. In order to avoid losing information when we use the identity \eqref{eq:multistep:sgz}, we will assume throughout that $\alpha$ is bounded away from $\partial \D$, that is, $\|\alpha\|_\infty < 1$.

\begin{theorem}[Two-Block Gordon Lemma] \label{t:2blgord}
Let $\alpha$ be bounded away from $\partial \D$. Suppose that there exist $z_0 \in \partial \D$, a constant $c > 0$, and a sequence of positive integers $n_k \to \infty$ such that
\begin{equation} \label{eq:2bl:rep}
\alpha(j)
=
\alpha(j+n_k),
\quad
0 \leq j \leq n_k - 1
\end{equation}
and $|\tr(T(n_k,0;z_0))| \leq c$ for all $k$. Then $z_0$ is not an eigenvalue of $\E_\alpha$. More precisely, there is a constant $\Delta > 0$ that depends only on $c$ and $\|\alpha\|_\infty$ such that 
\begin{equation} \label{eq:Phi:lb}
\max\big( \| \Phi(\underline{n_k}) \|, \| \Phi(2n_k) \| \big)
\geq
\Delta\|\Phi(0)\|
\end{equation}
for all $k$, where $\Phi(n) = (u_n,v_n)^\top$, $u$ solves $\E u = z_0 u$, $v = \mathcal M u$, and $\underline{n_k} = 2\lfloor n_k/2\rfloor$ denotes the greatest even integer which fails to exceed $n_k$.
\end{theorem}

If the Verblunsky coefficients exhibit an additional block of repetitions, one may abandon the hypothesis on the traces.

\begin{theorem}[Three-Block Gordon Lemma] \label{t:3blgord}
Suppose $\alpha$ is bounded away from $\partial \D$ and that there exist positive integers $n_k \to \infty$ such that $\alpha(j-n_k) = \alpha(j) = \alpha(j+n_k)$ for all $0 \leq j \leq n_k - 1$. Then $\E_\alpha$ has purely continuous spectrum.
\end{theorem}

Moreover, one need not insist that repetitions be exact in Theorem~\ref{t:3blgord}. We refer to sequences of Verblunsky coefficients that obey an approximate version of the three-block repetition property as Gordon sequences. More precisely, let us say that $\alpha \in \D^{\Z}$ is a \emph{Gordon sequence} if it is bounded away from $\partial \D$ and there exist positive integers $n_k \to \infty$ such that 
\begin{equation} \label{eq:gorddef}
\lim_{k\to\infty} 
C^{n_k} \max_{0 \leq j \leq n_k - 1}|\alpha(j) - \alpha(j \pm n_k)|
=
0
\end{equation}
for all $C > 0$. The $\pm$ in \eqref{eq:gorddef} is meant to signify that the equation is true for either choice of sign. These sequences also produce CMV operators with empty point spectrum.

\begin{theorem} \label{t:gordpot}
If $\alpha$ is a Gordon sequence, then $\E_\alpha$ has purely continuous spectrum.
\end{theorem}

\begin{remark}
For simplicity and clarity, we have stated the Gordon lemmas with repetitions starting at/centered around the origin. However, it is easy to see that one can suitably modify the statements to apply if one has repetitions which are centered around some other fixed integer $\ell_0 \in \Z$.
\end{remark}

The structure of the paper is as follows. In Section~\ref{sec:gord}, we prove the new Gordon Lemmas for CMV matrices, i.e. Theorems~\ref{t:2blgord}, \ref{t:3blgord} and \ref{t:gordpot}. Section~\ref{sec:sturm} uses the new Gordon Lemmas to prove Theorems~\ref{t:sturm:scspec} and \ref{t:rotcode}. Additionally, we provide a simple proof that the spectra of Sturmian CMV matrices may be characterized as the set of spectral parameters at which the associated trace map enjoys a bounded orbit. To the best of our knowledge, trace bounds in this setting have not been worked out in this level of generality. Moreover, we believe that our proof which establishes that traces are bounded on the spectrum (i.e., the proof of Lemma~\ref{l:zsubsetb}) is somewhat simpler than the standard arguments; we use growth estimates for escaping trace orbits \cite{DGLQ} and the fact that the Lyapunov exponent vanishes on the spectrum \cite{damaniklenz:szegob}.

\section{Proofs of Gordon Lemmas} \label{sec:gord}

We now prove the new Gordon Lemmas. To deal with odd repetition scales, we apply Cayley--Hamilton to the Szeg\H{o} transfer matrices, use the relationship \eqref{eq:multistep:sgz} over a slightly smaller block of even length, and then use that $\alpha$ is bounded away from $\partial \D$ to see that one still has effective lower bounds on the GZ evolution. The precise details follow presently.

\begin{proof}[Proof of Theorem~\ref{t:2blgord}]
Let $u \in \C^{\Z}$ be a solution to $\E u = z_0 u$, and define $v = \mathcal M u$, $\Phi(n) = (u_n,v_n)^\top$, and $\Phi = \Phi(0)$. Thus, one has $\Phi(n) = Z(n,0;z_0) \Phi$ for all $n \in \Z$. Without loss, assume $|u_0|^2 + |v_0|^2 = 1$. Since $\Theta(\alpha)$ is unitary for every $\alpha \in \D$, it follows that
\begin{equation} \label{eq:uPhiconv}
\lim_{n \to \infty} u_n = 0
\iff
\lim_{n \to \infty} v_n = 0
\iff
\lim_{n \to \infty} \Phi(n) = 0.
\end{equation}
We claim that \eqref{eq:Phi:lb} holds with $\Delta$ defined by
\begin{equation} \label{eq:gammadef}
\eta
=
\frac{1}{2} \min(1,c^{-1}),
\quad
\Gamma
=
\sup_{n \in \Z} \|S(\alpha_n,z)\|,
\quad
\Delta
=
\frac{\eta}{\Gamma}.
\end{equation}
Notice that $\Gamma < \infty$ (and hence $\Delta > 0$), since $\alpha$ is bounded away from $\partial \D$. Moreover, we also have $\Gamma \geq 1$ (and hence $\Delta \leq \eta$) since $\det(S(\alpha,z_0)) = z_0$ for all $\alpha \in \D$. If $n_k$ is even, one may apply the arguments from \cite{O14} without modification to deduce that
\[
\max(\|\Phi(2n_k)\| , \| \Phi(\underline{n_k}) \|)
=
\max(\|\Phi(2n_k)\| , \| \Phi(n_k)\|)
\geq
\eta
\geq
\Delta,
\]
since \eqref{eq:multistep:sgz} implies that $|\tr(Z(j,0;z))| = |\tr(T(j,0;z))|$ whenever $j$ is even and $z \in \partial \D$. Now, suppose $n_k$ is odd. Denote  $\Psi(n) = T(n,0;z_0) \Phi$ for $n \in \Z$. By Cayley--Hamilton and \eqref{eq:2bl:rep}, we have
$$
\Psi(2n_k) - \tr(T(n_k,0;z_0)) \Psi(n_k) + z_0^{n_k} \Phi
=
0,
$$
so, consequently, we deduce
$$
\max( \| \Psi(2n_k) \|, \| \Psi(n_k) \| )
\geq
\eta.
$$
If $\|\Psi(2n_k)\| \geq \eta$, we may apply the identity \eqref{eq:multistep:sgz} to see that
\[
\|\Phi(2n_k)\| 
=
\| z_0^{-n_k} T(2n_k,0;z_0) \Phi \|
=
\| \Psi(2n_k) \|
\geq
\eta
\]
as well. Otherwise, $\|\Psi(n_k)\| \ge \eta$. Notice that $\underline{n_k} = n_k-1$, which yields
\[
S(\alpha_{n_k-1},z_0) \Phi(\underline{n_k})
=
z^{-\frac{n_k-1}{2}} T(n_k,0;z) \Phi.
\]
Thus, $\| \Phi(\underline{n_k}) \| \geq \Gamma^{-1} \| \Psi(n_k) \| \geq \Delta$, which concludes the proof of \eqref{eq:Phi:lb}. By \eqref{eq:uPhiconv}, $u \notin \ell^2(\Z)$. Since no nontrivial solutions $u$ of $\E u = z_0 u$ are square-summable, we conclude that $z_0$ is not an eigenvalue of $\E$.
\end{proof}

\begin{proof}[Proof of Theorem~\ref{t:3blgord}]
Let $z \in \partial \D$ be given, and let $u$, $v$, $\Phi$, $\Psi$, and $\Gamma$ be as in the previous theorem. If $|\tr (T(n_k,0;z))| \leq 1$, then the proof of the previous theorem yields 
\[
\max\big( \|\Phi(\underline{n_k})\|, \| \Phi(2n_k) \| \big)
\geq
\Delta,
\] 
where $\Delta$ is defined by \eqref{eq:gammadef} with $c = 1$, i.e. $\Delta = \frac{1}{2\Gamma}$. Otherwise, suppose $|\tr (T(n_k,0;z))| > 1$. Arguing as in \cite{O14} and suitably modifying the argument as in the previous proof for odd $n_k$, we obtain
$$
\max(\| \Phi(-\underline{n_k}) \|, \| \Phi(\underline{n_k}) \| )
\geq
\frac{1}{2\Gamma}
>
0.
$$
More specifically, if $n_k$ is even, then the arguments of \cite{O14} give
\[
\max(\| \Phi(-\underline{n_k}) \|, \| \Phi(\underline{n_k}) \| )
\geq
\frac{1}{2}
\geq
\Delta.
\]
Notice that this uses \eqref{eq:multistep:sgz} again to deduce $|\tr(T(j,0;z))| = |\tr(Z(j,0;z))|$ for even $j$. Now, assume $n_k$ is odd. By Cayley--Hamilton and our assumptions on $\alpha$, we have
\[
\Psi(n_k) - \tr(T(n_k,0;z)) \Phi + z\Psi(-n_k)
=
0,
\]
which implies
\[
\max\big(\|\Psi(-n_k)\|, \|\Psi(n_k)\|\big)
>
\frac{1}{2},
\]
since $|\tr(T(n_k,0;z))| > 1$. Arguing as in the proof of Theorem~\ref{t:2blgord}, we use \eqref{eq:multistep:sgz} to obtain
$$
\max\big(\| \Phi(\underline{n_k}) \|, \| \Phi(-\underline{n_k}) \| \big)
\geq
\Gamma^{-1} \max\big(\| \Psi(n_k)\|, \| \Psi(- n_k)\| \big)
\geq
\frac{1}{2\Gamma}.
$$
Thus, we have shown that 
\begin{equation} \label{eq:3blgordbound}
\max\big( \|\Phi(-\underline{n_k})\|, \| \Phi(\underline{n_k})\|, \| \Phi(2n_k) \| \big)
\geq
\Delta
\end{equation}
for all $k$.
\end{proof}

For the next proof, let us note the following simple estimate. 

\begin{lemma} \label{l:gron}
For every $r \in (0,1)$, there is a constant $C = C(r)$ such that the following holds true: For any sequences $\alpha,\widetilde\alpha \in \D^{\Z}$ such that $\|\alpha\|_\infty, \| \widetilde \alpha\|_\infty \leq r$, and $|\alpha(j) - \widetilde \alpha(j)| < \delta$ for all $j$ between $n$ and $m$, one has
$$
\| Z(n,m;z) - \widetilde Z(n,m;z)\|
\leq
\delta C^{|n-m|}
$$
for all $n,m \in \Z$ and all $z \in \partial \D$. 
\end{lemma}

\begin{proof}
Since $|\det(Z(n,m;z))| = 1$ for all $n,m \in \Z$ and all $z \in \C\setminus\{0\}$, it suffices to consider $m = 0$ and $n > 0$. Notice that
\[
Z(n,0;z) - \widetilde Z(n,0;z)
=
\sum_{j=0}^{n-1} \left( \widetilde Z(n,j+1;z) Z(j+1,0;z) 
-
\widetilde Z(n,j;z) Z(j,0;z) \right).
\]
The desired estimate follows readily by writing the summands as
\[
\widetilde Z(n,j+1;z)(Y(j,z) - \widetilde Y(j,z) )Z(j,0;z).
\]
In particular, there exists $A = A(r) \in (1, \infty)$ such that
\begin{align*}
\|P(\alpha,z)\|
& \leq
A \\
\| Q(\alpha,z) \|
& \leq
A \\
 \|P(\alpha,z) - P(\beta,z)\|
& \leq
A|\alpha - \beta| \\
\|Q(\alpha,z) - Q(\beta,z)\|
& \leq
A|\alpha - \beta|
\end{align*}
for all $z \in \partial \D$ and all $\alpha,\beta \in \D$ with $|\alpha|, \, |\beta| \leq r$. Consequently,
\[
\|Z(n,0;z) - \widetilde Z(n,0;z)\|
\leq
n \delta A^n,
\]
which yields the desired result with $C = A\sqrt[3]{3}$.
\end{proof}

\begin{proof}[Proof of Theorem~\ref{t:gordpot}]
For each $k$, consider the $n_k$-periodic CMV matrix $\E_k$ whose Verblunsky coefficients are given by repeating the string $\alpha_0,\ldots,\alpha_{n_k-1}$ periodically, and consider $u$ and $\Phi$ as in the previous arguments. By \eqref{eq:gorddef} and Lemma~\ref{l:gron}, we deduce that
$$
\lim_{k \to \infty}
\|Z_k(2n_k,0;z) - Z(2n_k,0;z)\|
=
\lim_{k \to \infty}
\|Z_k(-\underline{n_k},0;z) - Z(-\underline{n_k},0;z)\|
=
0,
$$ 
where $Z_k$ denotes the family of GZ transfer matrices associated with $\E_k$. Applying \eqref{eq:3blgordbound} to $\E_k$, we may choose $\ell_k \in \set{-\underline{n_k}, \underline{n_k}, 2n_k}$ such that
\[
\| Z_k(\ell_k,0;z) \Phi\|
\geq
\Delta,
\]
where $\Delta = \frac{1}{2\Gamma}$ as in the proof of Theorem~\ref{t:3blgord}. Thus, we have
$$
\| \Phi(\ell_k) \|
\geq
\|Z_k(\ell_k,0;z)\Phi\| - \|Z(\ell_k,0;z)-Z_k(\ell_k,0;z) \|
\geq
\frac{\Delta}{2}
$$
for all sufficiently large $k$. Consequently, $\Phi$ is not in $\ell^2(\Z)$, so neither is $u$.
\end{proof}

\section{Application: Singular Continuous Spectrum for All Sturmian CMV Matrices} \label{sec:sturm}

In this section, we prove Theorem~\ref{t:sturm:scspec}. The key ingredients are the new two-block Gordon Lemma, trace bounds on the spectrum, and a combinatorial analysis of the structure of elements of the subshift $\Omega_\theta$. We begin by recalling a few highlights from the theory of continued fractions without proofs.  For a detailed exposition of the theory with proofs, we refer the reader to Khinchin's beautiful text \cite{khin}.

To any $\theta \in (0,1) \setminus \Q $, there is a unique assocated sequence $(a_j)_{j=1}^\infty$ of positive integers with
\begin{equation} \label{d:cfrac:exp}
\theta
=
[a_1,a_2,\ldots ]
=
\cfrac{1}{a_1 + \cfrac{1}{a_2 + \cfrac{1}{\ddots}}}
\end{equation}
Stopping the continued fraction expansion of $\theta$ after $n$ steps produces rational numbers $r_n=[a_1,\ldots,a_n] = p_n/q_n$.  If we assume that all of these fractions are in lowest terms (i.e. $\gcd(p_n,q_n)=1$), then one obtains the following recursive formulae for $p_n$ and $q_n$:
\begin{align}
\label{eq:cfrac:rec1}
p_0 = 0, \quad p_1 & = 1, \quad
p_n  = a_n p_{n-1} + p_{n-2} \\
\label{eq:cfrac:rec2}
q_0 = 1, \quad q_1 & = a_1, \quad
q_n  = a_n q_{n-1} + q_{n-2}.
\end{align}
Let us now observe how the continued fraction coefficients of $\theta$ influence the combinatorial structure of elements of the hull $\Omega_\theta$. Fix $\omega_0 = s_{\theta,\theta}$, and define a sequence of words $\set{v_n}_{n=-1}^\infty$ as follows:
\begin{equation} \label{eq:sturmwordrec}
v_{-1}
=
1,
\quad
v_0
=
0,
\quad
v_1
=
v_0^{a_1-1}v_{-1},
\quad
v_n
=
v_{n-1}^{a_n} v_{n-2},
\; n \geq 1.
\end{equation}
Concretely, each $v_n$ is an element of $\set{0,1}^*$, the free monoid with two generators, so the multiplication in \eqref{eq:sturmwordrec} is given by concatenation. By \cite[Lemma~1(b)]{BIST}, one has
\begin{equation}\label{eq:omegaw}
v_k 
= 
\omega_0|_{[0,q_k - 1]}
\text{ for all } k \geq 1.
\end{equation}
Note the change in normalization: \cite{BIST} considers $s_{\theta,0}$, while we consider $\omega_0 = s_{\theta,\theta}$, which enables us to begin the analysis at the origin instead of at $1$ (so that we can apply our version of the two-block Gordon Lemma as stated in Theorem~\ref{t:2blgord}).

The combinatorial structure of $\omega_0$ is naturally manifested in the associated transfer matrices. More specifically, for $\zeta \in \C$, define
\begin{align*}
M_{-1}(\zeta) 
& =
S(\gamma,\zeta) S(\beta,\zeta)^{-1}
=
\frac{1}{\rho_\beta \rho_\gamma}
\begin{pmatrix}
1 - \beta\overline\gamma & \overline\beta-\overline\gamma \\
\beta-\gamma & 1 - \overline\beta\gamma
\end{pmatrix}, \\
M_0(\zeta) 
& = 
\zeta^{-1/2} S(\beta,\zeta)
=
\frac{\zeta^{-1/2}}{\rho_\beta}
\begin{pmatrix} \zeta & - \overline\beta \\
-\beta \zeta & 1
\end{pmatrix},
\end{align*}
and, for $n \geq 1$, define
\begin{equation}\label{def:sturm:transmats}
M_{n}(\zeta) 
= 
\zeta^{-q_n/2} T_{\omega_0}(q_n,0;\zeta),
\quad
n \geq 1, \, \zeta \in \C,
\end{equation}
where $T_{\omega_0}$ is used to denote the family of Szeg\H{o} transfer matrices associated to $\alpha_{\omega_0}$, given by \eqref{eq:sturmalphas}. For each fixed $\zeta$, one chooses branches inductively to ensure that
\begin{equation} \label{eq:zetabranch}
\zeta^{-q_n/2}
=
\zeta^{-a_n q_{n-1}/2} \zeta^{-q_{n-2}/2}
\text{ for all } n \geq 2.
\end{equation}
Of course, $M_n(\zeta)$ also depends on $\beta$, $\gamma$, and $\theta$, but we suppress this for the sake of readability. These matrices obey a recursive relationship which comes from \eqref{eq:omegaw}.

\begin{lemma} \label{l:tmrec}
Fix $\zeta \in \partial \D$. We have
\begin{equation} \label{eq:sturm:mat:rec}
M_{n}(\zeta)
=
M_{n-2}(\zeta) M_{n-1}(\zeta)^{a_{n}}
\text{ for all } n \geq 1.
\end{equation}
\end{lemma}

\begin{proof} 
This follows immediately from \eqref{eq:sturmwordrec}, \eqref{eq:omegaw}, and \eqref{eq:zetabranch}.
\end{proof}

Now, since the transfer matrices $M_n$ obey the same recursive relationship that one observes in the Schr\"odinger case, the traces may be realized by sampling orbits of the same dynamical system, simply with different initial conditions. More concretely, if we denote 
\[
x_n
=
\frac{1}{2} \tr(M_{n-1}),
\quad
y_n
=
\frac{1}{2} \tr(M_n),
\quad
z_n
=
\frac{1}{2} \tr(M_n M_{n-1}),
\]
then there are polynomial maps $F_n: \C^3 \to \C^3$ such that
\[
F_n(x_n(\zeta),y_n(\zeta),z_n(\zeta))
=
(x_{n+1}(\zeta) ,y_{n+1}(\zeta), z_{n+1}(\zeta))
\]
for all $n \geq 0$ and all $\zeta \in \C$. Moreover, these polynomial maps enjoy a first integral given by the Fricke-Vogt invariant. More precisely,
\[
I
=
I(\zeta)
\eqdef
x_n(\zeta)^2 + y_n(\zeta)^2 + z_n(\zeta)^2 - 2x_n(\zeta) y_n(\zeta) z_n(\zeta) - 1
\]
is independent of $n$ (though it will depend on $\zeta$). This is explicitly computed in \cite{DMY1}. They find that
\[
I(\zeta)
=
\frac{\mathrm{Re}(\zeta)}{2\rho_\beta^2 \rho_\gamma^2}(\rho_\beta^2 + \rho_\gamma^2 - K)
+\frac{1}{4\rho_\beta^2 \rho_\gamma^2}( K^2 - 2K + 2\rho_\beta^2 + 2\rho_\gamma^2) - 1,
\] 
where $K = 2 - 2\mathrm{Re}(\beta\overline\gamma)$. One can now characterize the uniform spectrum $\Sigma \eqdef \sigma(\E_{\omega_0})$ as the dynamical spectrum, that is, the set of energies such that the corresponding trace map enjoys a bounded orbit. 

\begin{theorem} \label{t:sturm:tracebounds}
Define
\begin{align*}
B 
& = 
\{ \zeta \in \C : \tr(M_n(\zeta)) \text{ is a bounded sequence in } \C \}, \\
B'
& =
\{ \zeta \in \C : (x_n(\zeta), y_n(\zeta), z_n(\zeta)) \text{ is a bounded sequence in } \C^3 \}.
\end{align*}
Then $B' = \Sigma = B$.
\end{theorem}

\begin{remark}
By minimality and strong operator convergence, we know that $\sigma(\E_{\omega}) = \Sigma$ for all $\omega \in \Omega_\theta$.
\end{remark}

When $\theta = \frac{\sqrt 5 - 1}{2}$, which corresponds to $a_n \equiv 1$, Theorem~\ref{t:sturm:tracebounds} is precisely \cite[Theorem~12.8.3]{S2}. One can extend Simon's arguments to the general case in a reasonable fashion \`a la \cite{BIST}, though the details become more technical. However, there is a simpler argument using hyperbolicity and the characterization of bounded trace orbits from \cite{DGLQ}, which we shall describe presently for the reader's convenience. First, let us recall the \emph{Lyapunov exponent}, which is given by
\begin{equation} \label{le:def}
L(\zeta)
=
\lim_{n \to \infty} \frac{1}{q_n} \log \|M_n(\zeta)\|
=
\lim_{m \to \infty} \frac{1}{m} \log\|T_{\omega_0}(m,0;\zeta) \|,
\quad
\zeta \in \partial \D. 
\end{equation}
Strictly speaking, we should define $L$ as an average over the associated hull with respect to the unique shift-invariant measure, that is, we should define
\[
L(\zeta)
=
\lim_{n \to \infty} \frac{1}{n} \int_{\Omega_\theta} \! \log\|T_\omega(n,0;\zeta)\| \, d\mu(\omega), 
\]
where $\mu$ is the unique shift-invariant Borel probability measure on $\Omega_\theta$ and $T_\omega$ denotes a Szeg\H{o} transfer matrix associated to $\alpha_\omega$ defined as in \eqref{eq:sturmalphas}. However, by \cite[Proposition~3.1]{damaniklenz:szegob}, this is equivalent to \eqref{le:def}. We then consider $\ZL$, the set on which the Lyapunov exponent vanishes, i.e.,
\[
\ZL
=
\{\zeta \in \partial \D : L(\zeta) = 0 \}.
\]
By \cite[Theorem~1.1]{damaniklenz:szegob}, $\Sigma = \ZL$. Thus, to prove Theorem~\ref{t:sturm:tracebounds}, it suffices to show that $B' =\ZL = B$. We shall do this in three steps, namely, we will show $B' \subseteq \ZL \subseteq B \subseteq B'$. 

First, we shall show that there are no decaying solutions to the generalized eigenvalue equation in $B'$. Since this precludes uniform hyperbolicity of the Gesztesy-Zinchenko cocycle \cite{DFLY2} and nonuniform hyperbolicity is absent in the present setting \cite{damaniklenz:szegob},  one can use the DFLY formula to see that this yields $B' \subseteq \ZL$ as an immediate byproduct. Moreover, once we prove $\Sigma = B'$, we will also have established absence of eigenvalues on the spectrum.

\begin{lemma} \label{l:bsubsetz}
Let $\E = \E_{\omega}$ and $\zeta \in B'$ be given. If $u$ is a nontrivial solution to the difference equation $\E u = \zeta u$, then $u_n$ does not converge to zero as $|n| \to \infty$. In particular, $B' \subseteq \ZL$.
\end{lemma}

\begin{proof}
Let $\zeta \in B'$ be given. Since we have bounds on traces, it remains to be seen that one can produce suitable repetitions in $\omega$ in order to apply the two-block Gordon lemma. The relevant combinatorial analysis is furnished by the proof of \cite[Lemma~4.1]{DKL2000}. More specifically, they prove that, for every $k \geq 3$, there exists a word $s = s_k \in \{ v_{k-1}, \, v_k, \, v_{k+1}, \, v_{k+1} v_k\}$ of length $n_k$ such that $\omega(0),\ldots,\omega(2n_k-1)$ is a cyclic permutation of $ss$. Thus, since $\zeta \in B'$, we know that $\tr(T_\omega(n_k,0;\zeta))$ is uniformly bounded as $k \to \infty$ by cyclicity of the trace. In particular, Theorem~\ref{t:2blgord} implies that there are no nontrivial solutions $u$ of $\E u = \zeta u$ such that $u_n \to 0$ as $|n| \to \infty$.

To prove the second statment, notice that the Gesztesy-Zinchenko cocycle cannot exhibit nonuniform hyperbolicity by \cite[Proposition~3.1]{damaniklenz:szegob}. Additionally, by \cite{DFLY2}, it cannot exhibit uniform hyperbolicity (as uniform hyperbolicity would imply the presence of decaying solutions to the difference equation). Thus, $L(\zeta) = 0$ by~\eqref{eq:multistep:sgz}.
\end{proof}

Next, we prove $\ZL \subseteq B$. Essentially, if $\zeta \notin B$, then \cite[Lemma~4.4]{DGLQ} establishes lower bounds on the growth of $\tr(M_n(\zeta))$, which yields an obvious lower bound on the growth of $\|M_n(\zeta)\|$, and these lower bounds suffice to prove $L(\zeta) > 0$.

There is one small subtlety to address regarding the invocation of \cite[Lemma~4.4]{DGLQ} in the following proof. Since \cite{DGLQ} works in the Schr\"odinger setting, they have (and use) the \textit{a priori} estimate $|\tr(M_{-1})| \leq 2$, which is false in the present setting for some values of $\beta$ and $\gamma$. However, if one replaces the condition $\delta \ge 0$ with the condition $\delta \ge \max(0,|\tr(M_{-1})| -2)$ in the statement of \cite[Lemma~4.4]{DGLQ}, the the proof of the same carries through verbatim to the present setting.

\begin{lemma} \label{l:zsubsetb}
Let $\beta$, $\gamma$, and $\theta$ be given. Then $\ZL \subseteq B$.
\end{lemma}

\begin{proof}
Suppose $\zeta \in \partial \D \setminus B$. Then, by \cite[Lemma~4.4]{DGLQ}, there exists $\delta > 0$ and $k_0 \in \Z_+$ such that
\begin{equation} \label{eq:tracegrowth}
|y_{k_0+k}(\zeta)|
\geq
\frac{1}{2} (1+\delta)^{G_k^{(k_0)}}
\text{ for all } k \geq k_0,
\end{equation}
where
\begin{equation} \label{eq:Gdef}
G_0^{(k_0)}
=
1,
\quad
G_1^{(k_0)}
=
a_{k_0+1},
\quad
G_{k+1}^{(k_0)}
=
a_{k_0+k+1} G_k^{(k_0)} + G_{k-1}^{(k_0)},
\;
k \geq 1.
\end{equation}
Notice that \cite{DGLQ} works in the Schr\"odinger setting, but this only amounts to a change in initial conditions. More specifically, our transfer matrices obey \eqref{eq:sturm:mat:rec}, which is precisely the same recursion as the Schr\"odinger matrices, and this is all that is used in the proof of \cite[Lemma~4.4]{DGLQ}, so their characterization of escaping trace orbits applies to the present setting. Comparing \eqref{eq:Gdef} to \eqref{eq:cfrac:rec2}, we see that there is a constant $C > 0$ such that
\begin{equation} \label{eq:Gqbound}
\frac{G_k^{(k_0)}}{q_{k + k_0}}
\geq
C
\text{ for every } k \geq 0.
\end{equation}
More specifically, define 
\[
C 
= 
\min\!\left(q_{k_0}^{-1}, a_{k_0 + 1}/q_{k_0 + 1} \right)
\] 
so that \eqref{eq:Gqbound} holds for $k = 0$ and $k = 1$. One then obtains \eqref{eq:Gqbound} for all $k$ by induction. Consequently, we have 
\[
L(\zeta)
\geq
C \log(1+\delta) 
> 
0\] 
by \eqref{eq:tracegrowth} and \eqref{eq:Gqbound}.
\end{proof}

Finally, we prove $B \subseteq B'$. This is almost immediate from the conservation of the Fricke-Vogt character:
\begin{equation} \label{eq:fvinv}
x_n(\zeta)^2 + y_n(\zeta)^2 + z_n(\zeta)^2 - 2x_n(\zeta) y_n(\zeta) z_n(\zeta) -1 
=
I(\zeta)
\text{ for all } n \geq 0.
\end{equation}

\begin{lemma} \label{l:bsubsetb'}
With setup as above, $B \subseteq B'$.
\end{lemma}

\begin{proof}
Let $\zeta \in B$ be given. To show $\zeta \in B'$, it suffices to show that $z_n(\zeta)$ is a bounded sequence. However, we can solve \eqref{eq:fvinv} for $z_n$ in terms of $(x_n,y_n)$ and then use boundedness of the latter to deduce boundedness of the former.
\end{proof}

\begin{proof}[Proof of Theorem~\ref{t:sturm:tracebounds}]
This is immediate from Lemmas~\ref{l:bsubsetz}, \ref{l:zsubsetb}, and \ref{l:bsubsetb'}.
\end{proof}

\begin{proof}[Proof of Theorem~\ref{t:sturm:scspec}]
By Lemma~\ref{l:bsubsetz}, $\E_\omega$ has no eigenvalues in $B'$. Since $\Sigma = B'$ by Theorem~\ref{t:sturm:tracebounds}, we are done.
\end{proof}

We now turn towards codings of rotations. Unlike the Sturmian setting, the techniques of \cite{O14} apply without any modification.

\begin{proof}[Proof of Theorem~\ref{t:rotcode}]
This follows immediately from the arguments that prove \cite[Theorem~4]{O14}. Simply replace \cite[Proposition~2]{O14} by Theorem~\ref{t:3blgord}. In particular, it follows from the arguments of \cite{Kam96} that the set of phases for which the associated Verblunsky coefficients satisfy the assumptions of three-block Gordon Lemma (Theorem~\ref{t:3blgord}) has Lebesgue measure at least
\[
\limsup_{n \to \infty} \left( 1 - \frac{4q_n}{q_{n+1}} \right),
\] 
and \cite[Lemma~4]{Kam96} implies that
\[
\limsup_{n \to \infty} \frac{q_{n+1}}{q_n}
\geq
\frac{1}{2}\left[\limsup_{n \to \infty} a_n + \sqrt{\left(\limsup_{n \to \infty} a_n \right)^2 + 4} \right].
\]
This yields the desired statement by almost-sure invariance of the pure point spectrum \cite[Theorem~10.16.1]{S2}.
\end{proof}

\section*{Acknowledgements} J.\ F.\ is grateful to David Damanik for helpful comments. J.\ F.\ was supported in part by NSF grants DMS--1067988 and DMS--1361625.

\end{document}